\documentclass{amsart}[12pt font]
\author{Tamanna Chatterjee, Laura Rider}
\date{\today}
\title{Action of the Relative Weyl group on partial springer sheaf}
\usepackage{amssymb}
\usepackage{amsmath}
\usepackage{calligra}
\usepackage{amsfonts}
\usepackage{mathrsfs}
\usepackage{mathtools}
\usepackage{tikz-cd}
\usepackage{pgfkeys,pgfopts,xcolor}
\usepackage{ytableau}
\usepackage{placeins}
\usetikzlibrary{matrix}
\usepackage{geometry,mathtools}
\theoremstyle{plain}
\newtheorem{theorem}{Theorem}[section]
\newtheorem{cor}[theorem]{Corollary}
\newtheorem{lemma}[theorem]{Lemma}

\newtheorem{ex}[theorem]{Example}

\newcommand{\mf}{\mathfrak}
\newcommand{\mc}{\mathcal}
\newcommand{\mb}{\mathbb}

\newcommand{\eal}{\end{align*}}

\newcommand{\bal}{\begin{align*}}

\DeclareMathOperator{\inn}{Ind}
\DeclareMathOperator{\p}{Perv}
\DeclareMathOperator{\f}{For}

\DeclareMathOperator{\loc}{Loc_f}

\DeclareMathOperator{\rhm}{R\mathscr{H}\text{\kern -3pt{\calligra\large om}}\, }
\DeclareMathOperator{\lie}{Lie}
\DeclareMathOperator{\h}{H}
\DeclareMathOperator{\Hom}{Hom}

\DeclareMathOperator{\IC}{\mathcal{IC}}

\DeclareMathOperator{\kp}{{\mathit{k}}}
\DeclareMathOperator{\NP}{\widetilde{\mathcal{N}}^P}
\DeclareMathOperator{\N}{\widetilde{\mathcal{N}}}
\DeclareMathOperator{\G}{\widetilde{\mathfrak{g}}}
\DeclareMathOperator{\GP}{\widetilde{\mathfrak{g}}^P}
\DeclareMathOperator{\grs}{\mf{g}_{rs}}
\DeclareMathOperator{\GRSP}{\widetilde{\mf{g}}^P_{rs}}
\DeclareMathOperator{\GRS}{\widetilde{\mf{g}_{rs}}}
\DeclareMathOperator{\LP}{L^P_{rs}}

\DeclareMathOperator{\End}{End}
\DeclareMathOperator{\Sp}{\underline{Spr}}
\DeclareMathOperator{\Spp}{\underline{Spr}^P}
\DeclareMathOperator{\Gr}{\underline{Groth}}
\DeclareMathOperator{\GrP}{\underline{Groth}^P}
\DeclareMathOperator{\pirsp}{\pi^P_{rs}}
\DeclareMathOperator{\con}{con}
\DeclareMathOperator{\pirs}{\pi_{rs}}
\DeclareMathOperator{\sgn}{sgn}
\DeclareMathOperator{\SpL}{\underline{Spr_L}}
\DeclareMathOperator{\ph}{^p H}
\DeclareMathOperator{\pd}{^{p}\mathcal{D}}
\DeclareMathOperator{\pt}{^p \tau}
\DeclareMathOperator{\SpU}{\underline{Spr}^U}
\DeclareMathOperator{\rhoU}{\phi_{U}}
\DeclareMathOperator{\phiU}{\rho_U}
\begin{document}
\maketitle
\begin{abstract} In the context of the Springer correspondence, the Weyl group action on the Springer sheaf can be defined in two ways: via restriction or the Fourier transform. It is well-known that these two actions differ by the sign character. This was proven by Hotta for sheaves with characteristic 0 coefficients in 1981, and more recently extended to arbitrary coefficients by Achar, Henderson, Juteau, and Riche in 2014. In this short article, we study an extension of this problem to the partial Springer sheaf with arbitrary field coefficients. This involves an action of the so-called relative Weyl group $W(L)$. 
\end{abstract}
\section{Introduction} Let $G$ be a complex, connected, reductive algebraic group, and let $\mc{N}$ be its nilpotent cone. The nilpotent cone is a singular variety, and so we study it together with its resolution of singularities $\N \to \mc{N}$, known as the Springer resolution. The geometry of this map appears throughout Lie theoretic representation theory. It is interesting and sometimes advantageous to consider resolving these singularities ``one simple root at a time." More precisely, each parabolic subgroup $P$ in $G$ gives rise to a space $\NP$ with maps,
\[ \N \to \NP \to \mc{N},
\]
whose composition equals the Springer resolution $\N \to \mc{N}$. The space $\NP$ ``partially resolves" the singularties of $\mc{N}$ in the sense that the map $\NP \to \mc{N}$ is proper, birational, and an isomorphism away from the singular points of $\mc{N}$, but $\NP$ isn't smooth unless $P$ is a Borel subgroup of $G$.

The space $\NP$ has been used to define parabolic induction and restriction functors for perverse sheaves on the nilpotent cone. This is most useful in the context of the generalized Springer correspondence \cite{Lu3}, \cite{AJHR}, \cite{AJHR2}, \cite{AJHR3} for grouping simple perverse sheaves on the nilpotent cone according to ``cuspidal data." On the other hand, one could study induction of the constant sheaf, which we refer to as the \textit{partial Springer sheaf} $\Spp$. This perverse sheaf, as well as a couple of similarly defined variants, denoted $\SpU$ and $\GrP$, will be the main players of this article. 
 
It is well-known in Springer theory that there are two ways to define the Springer sheaf, $\Sp$, and each inherits an action of the Weyl group $W$. For both constructions, we must utilize the Grothendieck--Springer sheaf $\Gr$ which lives on $\frak{g}$, the Lie algebra of $G$. The first construction, due to Springer in \cite{Sp} and later revisited by Lusztig with the language of perverse sheaves \cite{Lu1}, is by restriction to $\mc{N}$. The second is by applying the Fourier transform. It was proved in \cite{Ho} that the $W$-action obtained on $\Sp$ in these two constructions differ by the sign character for sheaves with characteristic 0 coefficients. 
 
The extension of these results to sheaves with positive characteristic coefficients requires a slightly different approach. In particular, it is not as straightforward to prove that the $W$-action from restriction on $\Sp$ induces an isomorphism $\kp[W]\cong\End(\Sp)$. This is proven directly by Riche in an unpublished article \cite{Ri} with some restrictions on the characteristic of $k$. On the other hand, since Fourier Transform is an auto-equivalence, proof of the isomorphism $\kp[W]\cong\End(\Sp)$ induced by Fourier transform is uniform for all characteristics. This was carried out by Juteau in \cite{Ju}. With an aim to bridge the disparity, Achar, Juteau, Henderson, and Riche prove that the sign change theorem holds with arbitrary coefficients \cite[Theorem 1.1]{AJHR}. Then they may utilize the isomorphism already known from the Fourier Transform to prove as an easy corollary that restriction also induces an isomorphism $\kp[W]\cong\End(\Sp)$ for arbitrary coefficients.

In this note, we study an analogue of the sign change theorem for the \textit{partial Springer sheaf} $\Spp$. Similar to the Springer case, $\Spp$ can be obtained by restriction of $\GrP$, the partial Grothendieck--Springer sheaf, from $\frak{g}$ to $\mc{N}$. However, unlike the Springer case, the Fourier Transform of $\GrP$ is not $\Spp$. It is a different perverse sheaf which we denote $\SpU$. 

Each of $\Spp, \GrP,$ and $\SpU$ admits an action of the \textit{relative Weyl group} $W(L):= N_G(L)/L$, where $L$ is the Levi factor of $P$. Our comparison of $W(L)$-actions occurs on $\SpU$. More specifically, we define maps
\[\phiU: \kp[W(L)]\to \End(\SpU) \text{, defined by restriction,}\]
and
\[\rhoU: \kp[W(L)]\to \End(\SpU) \text{, defined by the Fourier transform.} 
\] Note that the group $W(L)$ is not always a Coxeter group, so the sign representation isn't necessarily available. The analogue of \cite[Theorem 1.1]{AJHR} proven in this paper is the following theorem. 
\begin{theorem}
	There exists a map $\Lambda_U: \kp[W(L)]\to \kp[W(L)]$ such that,
	\[\rhoU=\phiU\circ \Lambda_U.\] 
\end{theorem}

\subsection{Outline}
In Section \ref{prelim}, we go through some necessary background and notation. Here we discuss the Springer sheaf and the partial Springer sheaf. In Section \ref{sec3}, we define another perverse sheaf $\SpU$. In Section \ref{sec4}, we define parabolic induction and show that the partial Springer sheaf is a sub-object of the Springer sheaf. In Section \ref{sec5}, we define two actions of the relative Weyl group $W(L)$ on $\SpU$ and prove that these two actions are associated by the top cohomology of the partial flag variety.

\subsection{Acknowledgement}
The authors would like to express their gratitude to William Graham, Daniel Nakano, Pramod N. Achar and Daniel Juteau for their helpful advice and comments.

\tableofcontents
\section{Preliminaries}\label{prelim}
In this article, we are working with constructible sheaves on an algebraic variety. The algebraic varieties considered are defined over $\mathbb{C}$. We'll study sheaves of $k$-modules, where $k$ denotes a commutative Noetherian ring of finite global dimension. Given an action of an algebraic group $H$ on such a variety $X$, let $\p_H(X)\subset D^b_H(X)$ denote the categories of $H$-equivariant perverse sheaves on $X$ and the bounded $H$-equivariant derived category of sheaves on $X$ as in \cite{BL}. Define $\Hom^i(\mc{F},\mc{G}):=\Hom_{D^b_H(X)}(\mc{F},\mc{G}[i])$ for $\mc{F},\mc{G}\in D^b_H(X)$. Let $\underline{k}_X$ denote the rank one constant sheaf on $X$. For a stratified space $X$ with stratum $X_\lambda$ and local system $\mc{L}$ on $X_\lambda$, $\mc{IC}(X_\lambda,\mc{L})$ denotes the corresponding intersection cohomology complex suitably shifted so that it is perverse. This is a simple perverse sheaf when $k$ is a field and $\mc{L}$ is irreducible. All simple perverse sheaves on $X$ arise in this way when $k$ is a field.

\subsection{Reminder on Springer Theory}
Fix a complex connected reductive group $G$, and let $B$ denote a Borel subgroup of $G$ with unipotent radical $U$. We define $\N:= G\times^B \mf{u}$ where $\mf{u}$ is the Lie algebra of $U$. The space $\N$ can also be identified with \[\N \cong \{(gB,x)\in G/B\times \mc{N}| \hspace{2mm} \mathrm{Ad}(g^{-1})x\in \mf{u}\}.\] 
This space comes with a map $\mu: \N \to \mc{N}$ defined by $\mu(gB,x)=x$ which is called the Springer resolution. We also define $\G:=G\times^B \mf{b}$ which identifies with $\{(gB,x)\in G/B\times \mf{g}| \hspace{2mm} \mathrm{Ad}(g^{-1})x\in \mf{b}\}.$ Similar to $\mu$, the Grothendieck simultaneous resolution map is defined by 
\[\pi: \G \to \mf{g},\hspace{1cm} \pi(gB,x)=x.\] We denote the collection of all regular semi-simple elements in $\mf{g}$ by $\grs$, $\GRS=\pi^{-1}(\grs)$, and set $\pi|_{\GRS}=\pirs: \GRS \to \grs$. As usual, $W$ denotes the Weyl group of $G$. These spaces and maps fit into a commutative diagram:
\begin{equation}\begin{tikzcd}\label{eq:GSdiag}
\N  \arrow[d, "\mu"]  \arrow[r, hook] & \G \arrow[d, "\pi"] & \GRS \arrow[d, "\pirs"] \arrow[l, hook'] \\
\mc{N} \arrow[r, hook, "i_{\mc{N}}"] & \mf{g} & \grs  \arrow[l, hook']   
\end{tikzcd}\end{equation}
with left- and right-hand squares being Cartesian.

The following standard lemma plays an important role in Springer theory.
\begin{lemma}\label{lm2.1}
\begin{enumerate}
\item The map $\pirs: \GRS \to \grs$ is a Galois covering map with Galois group $W$. 
\item The map $\pi: \G\to \mf{g}$ is small with respect to $\grs$.
\end{enumerate}
\end{lemma}

Since $\pirs$ is a covering map, $\pirs_* \underline{\kp}_{\GRS}$ is locally constant \cite{BM}. Furthermore, since the Galois group of $\pirs$ is $W$, the local system $\pirs_* \underline{\kp}_{\GRS}$ corresponds to the regular representation of $W$, and so also has endomorphism ring isomorphic to $\kp[W]$. On the other hand, part (2) of the lemma means that $\pi_*\underline{\kp}_{\tilde{\mf{g}}}[\dim \mf{g}]$ is an IC-sheaf. Now, Proper Base Change Theorem using the right hand side of Diagram \eqref{eq:GSdiag} implies that $\pi_*\underline{\kp}_{\tilde{\mf{g}}}[\dim \mf{g}]\cong \IC(\GRS, \pirs_*\underline{\kp}_{\grs})$. The perverse sheaf $\pi_*\underline{\kp}_{\tilde{\mf{g}}}[\dim \mf{g}]$ is commonly called the Grothendieck--Springer sheaf, and we denote it by $\Gr$. Finally since the IC-functor is fully faithful, we get an isomorphism \begin{equation}\label{eq:endGroth} \kp[W]\cong \End(\Gr).\end{equation} This is carefully explained in \cite{Ac}, but before that it was observed by \cite{Sp}, \cite{Lu}. 

The primary application of this setup is the study of $\Sp:= \mu_* \underline{\kp}_{\N}[\dim \mc{N}]$. Because the map $\mu$ is semi-small (see for instance \cite[Lemma ~8.1.6]{Ac}), $\Sp$ is perverse, and we call it the Springer sheaf. The relation between $\Gr$ and $\Sp$, along with properties of $\Gr$ are used to prove the following theorem.

\begin{theorem}[The Springer Correspondence by Restriction]  \label{th1.3} There are isomorphisms

 \begin{enumerate} 
 \item $\Sp\cong i_{\mc{N}}^* \Gr$ and 
 \item $\kp[W] \cong \End(\Sp)$ induced by $i_{\mc{N}}^*$.
 \end{enumerate}
 \end{theorem} 
 \noindent We denote the second isomorphism in Theorem \ref{th1.3} by \[\rho: \kp[W]\xrightarrow{\sim} \End(\Sp).\] 

 The first part of this theorem is an easy application of Proper Base Change Theorem using the left-hand square of Diagram \eqref{eq:GSdiag}. The second part of this theorem requires a little more explanation. In case $\kp$ is a field of characteristic 0, this is well-known and due to \cite{BM1}. The argument relies on the fact that $\mathrm{H}^\bullet(G/B)$ is isomorphic to the regular representation of $W$ which is a faithful representation of $W$. This implies that the induced map from $\kp[W]$ to $\End(\Sp)$ must be injective. Then a dimension calculation forces it to be an isomorphism. This line of reasoning also works when $\kp = \mathbb{Z}$ or when $\kp$ is a field with characteristic not dividing $|W|$. This is explained in \cite[Section 4]{AJHR1}, and is referred to as the `Easy' case.

 However, when $\kp$ is a field with characteristic dividing $|W|$, this approach falls short. In particular, $\mathrm{H}^\bullet(G/B)$ is not faithful as a $W$-representation in this case. In \cite{Ri}, Riche takes an alternative approach by interpreting $\End(\Sp)$ in terms of Borel--Moore homology. However, his proof still requires some restrictions on the characteristic of $\kp$. The proof in \cite{AJHR1} overcomes this difficulty by utilizing the Fourier transform which we discuss in the next section. 

\subsection{Fourier--Sato transform}
 In the context of Springer theory, Fourier transform has been studied in \cite{HK}, \cite{Lu4}, and \cite{Mi} 
 for characteristic 0 coefficients, and was initiated in \cite{Ju} for positive characteristic coefficients. In the modular Springer correspondence \cite{AJHR1}, \cite{AJHR2}, \cite{AJHR4}, Fourier transform has played an important role. The proof of the isomorphism in Theorem \ref{th1.3}, part (2) is not that straight-forward in positive characteristic. \cite{Bry} and \cite{HK}  initiated the study of the Springer theory using Fourier transform. 
 
A detailed discussion of the Fourier--Sato transform can be found in \cite[Subsection 2.7]{AJHR1}. For a vector bundle $X$, we denote the Fourier transform on $X$ by $T_{X}$. This functor has as input $D^b_{\con}(X,\kp)$, the subcategory of \textit{conic} objects in the derived category on $X$. The output is a conic object in the derived category of the dual vector bundle $X^*$. Moreover, $T_{X}$ is an equivalence of categories between the subcategories of conic objects. In this paper, $X$ will either be $\mathfrak{g}$ or a vector bundle over a (partial) flag variety. The Killing form allows us to identify $\mf{g}^*$ with $\mf{g}$, and so we regard $T_\mf{g}$ as an endo-functor on $D^b_{\con}(\mf{g},\kp)$. In this paper we have also used Fourier transform on $G\times^P \mf{g}$ denoted by $T_{G\times^P {\mf{g}}}$.

The following theorem \cite[Th. 6.2.4]{Ju} once again provides a connection between the sheaves $\Sp$ and $\Gr$.

\begin{theorem}\label{th1.4}
There is an isomorphism
\[T_\mf{g}(\Gr) \cong {i_{\mc{N}}}_*\Sp.\]\end{theorem} 
\noindent By the isomorphism \eqref{eq:endGroth} and the fact that ${i_{\mc{N}}}_*$ is fully-faithful, we have
\[\kp[W]\cong\End(\Gr)\xrightarrow {T_\mf{g}} \End({i_{\mc{N}}}_* \Sp)\cong \End(\Sp).\]
Since $T_{\mf{g}}$ is an equivalence of categories, the middle arrow in the above diagram is an isomorphism. Thus, we have an isomorphism $\phi:\kp[W]\to\End(\Sp)$ induced by Fourier transform. Now we have a map $\epsilon : W\to \{\pm 1\}$ which sends $w$ to $(-1)^{\ell(w)}$, where $\ell$ is the length of $w$. We can define the following map,
\[\sgn: \kp[W]\to \kp[W],\] by $\sgn(w)=\epsilon(w)w$. 

Following the construction of Springer \cite{Sp}, in \cite{Ju} the action of $\phi$ has been studied  over a local ring or finite field. On the other hand, following the construction of Lusztig \cite{Lu}, Achar, Henderson, Riche \cite{AHR} have worked with the other $W$-action coming from $\phi$.
 Relation between $\phi$ and $\rho$ has been studied by Achar, Henderson, Juteau and Riche. They have proved that these two actions of $W$ are related by the sign change \cite[Theorem ~1.1]{AJHR1}.
 
\begin{theorem}\label{th1.5}
	We have $\rho=\phi \circ \sgn$.
\end{theorem}

The above theorem together with Theorem \ref{th1.4} proves that the map $\rho$ induced by the restriction is an isomorphism for any field.

\subsection{The partial Springer resolution} Diagram \eqref{eq:GSdiag} factors according to a choice of parabolic subgroup $P$ containing $B$ with notation and maps to be defined below. 
  
  \begin{equation}\begin{tikzcd}\label{eq:PGSdiag}
\N  \arrow[d, "\zeta"]  \arrow[r, hook] & \G \arrow[d, "\zeta'"] & \GRS \arrow[d, "\zeta''"] \arrow[l, hook'] \\
\NP  \arrow[d, "\mu_P"]  \arrow[r, hook] & \GP \arrow[d, "\pi_P"] & \GRSP \arrow[d, "\pirsp"] \arrow[l, hook'] \\
\mc{N} \arrow[r, hook, "i_{\mc{N}}"] & \mf{g} & \grs  \arrow[l, hook']   
\end{tikzcd}\end{equation} Let $P$ be a parabolic subgroup of $G$ containing $B$ with $L$ a Levi subgroup of $P$ and $U_P$ the unipotent radical. Let $W_L$ be the Weyl group associated with $L$. The nilpotent cone inside $\mf{p}=\mathrm{Lie}(P)$ will be denoted by $\mc{N}_P$, and $\mc{P}$ denotes the collection of all parabolic subgroups conjugate to $P$. Define $\NP=G\times^P \mc{N}_P$, which can also be thought as $\{(gP,x)\in G/P\times \mc{N}| \mathrm{Ad}(g^{-1})x\in \mc{N}_{P} \}$. The partial resolution is defined by  
$\mu_P: \NP \to \mc{N}$ with
\[ \mu_P(gP,x)=x.\] This map is proper. 
We define $\Spp:= {{\mu_P}_*}\underline{\kp}_{\NP}[\dim \mc{N}]$, which will be called partial Springer sheaf.
Similarly, we define $\GP:= G\times^P \mf{p}$ which identifies with $\{(gP,x)\in G/P\times \mf{g}|\mathrm{Ad}(g^{-1})x\in \mf{p}\}$. The partial simultaneous resolution is the map $\pi_P: \GP \to \mf{g}$ with
\[ \pi_P(gP,x)=x.\] The partial Grothendieck sheaf is $\GrP:= {\pi_P}_* \underline{\kp}_{\GP}[\dim \mf{g}] $. 

Recall $\grs$ denotes the set of all regular semisimple elements in $\mf{g}$. We let $\GRSP:=\pi_P^{-1}(\grs)$. The map $\zeta: \N\to \NP$ is defined by $\zeta(gB,x)=(gP,x)$. This is well-defined as $\mathrm{Ad}(g^{-1})x\in \mf{u}\subset \mc{N}_L+\mf{u}_P$. Clearly $\mu=\mu_P\circ \zeta$. Similarly we can define $\zeta': \G\to \GP$ so that $\pi=\pi_P\circ \zeta'$. Let $\zeta''$ be the map $\zeta'|_{\GRS}:\GRS\to \GRSP$. We denote the finite group $N_G(L)/L$ by $W(L)$, and call it the relative Weyl group even though it is not necessarily a Coxeter group. From \cite[Lemma ~5.2]{Lu1}, $W(L)$ identifies with
 \[W(L)=\{s\in N_W(W_L)| s(\Phi_L^+)=\Phi_L^+\},\] where $\Phi_L^+$ is the set of positive roots of $L$. Moreover $N_W(W_L)$ is a semi-direct product of the subgroups $W(L)$ and $W_L$.

The next lemma, an analogue of Lemma \ref{lm2.1}, and its proof have appeared in \cite[Lemma 2.5]{BM}. Note that Borho--MacPherson use notation which conflicts with our choices for the finite groups $W(L) = N_G(L)/L$ and $W_L = N_L(T)/T$. 
\begin{lemma}\label{lm1.6}
 	The map $\pi_P|_{\GRSP}=\pirsp: \GRSP\to \grs$ is a Galois covering map with Galois group $W(L)$.
 \end{lemma}

This means we have a surjective map $\pi_1(\grs)\twoheadrightarrow W(L)$. Therefore we can define the functor
\[\mathrm{L^P_{rs}}: \kp[W(L)]-\mathrm{mod^{fg}}\to \loc(\grs,\kp),\] where $\loc(\grs,\kp)$ is the category of finite type local systems on $\grs$ and $\kp[W(L)]-\mathrm{mod^{fg}}$ denotes the category of finitely generated $\kp[W(L)]$-modules. As $\pirsp$ is a Galois covering, 
${\pi^P_{rs}}_*\underline{ \kp}_{\GRSP}$ is a local system on $\grs$.

\begin{lemma}\label{lm2.7}
We have ${\pi^P_{rs}}_*\underline{\kp}_{\GRSP}\cong \LP(\kp[W(L)])$ and $\End({\pi^P_{rs}}_*\underline{\kp}_{\GRSP})\cong \kp[W(L)]$.
\end{lemma}

\begin{proof} The first isomorphism follows from Lemma \ref{lm1.6}. 
Using \cite[Prop.~ 1.7.16]{Ac},
  \[
  \End({\pi^P_{rs}}_*\underline{\kp}_{\GRSP})\cong \End(\LP\kp[W(L)]) \cong \End_{\kp[W(L)]}(\kp[W(L)] ) \cong  \kp[W(L)].
  \qedhere\]
  \end{proof}
The map $\pi_P$ is small with respect to $\grs$. The proof is similar to \cite[Lemma ~8.2.5]{Ac}. Therefore we get
\[{\pi_P}_* \underline{\kp}_{\GP}[\dim \mf{g}]\cong \mc{IC}(\grs, {\pi^P_{rs}}_*\underline{\kp}_{\GRSP}) 
\] by \cite[Prop. 3.8.7]{Ac}.
We will denote $ {\pi_P}_* \underline{\kp}_{\GP}[\dim \mf{g}] $ by $\GrP$.
From the above isomorphism we get  the next theorem.

\begin{theorem}\label{th2.8}
There is an isomorphism of rings
\[ \End( \GrP) \cong \kp[W(L)].
\]	
\end{theorem} 
Denote ${\mu_P}_*\underline{\kp}_{\NP}[\dim \mc{N}]$ by $\Spp$. 
By applying Proper Base Change Theorem with the bottom half of Diagram \eqref{eq:PGSdiag}, we get an isomorphism \[\Spp\cong i_{\mc{N}}^*\GrP.\] Together with Theorem \ref{th2.8}, this yields a map
\begin{align}\label{eq2.4}
  \kp[W(L)]\to \End(\Spp).  
\end{align}

\subsection{$W_L$-invariants} Using the isomorphism \eqref{eq:endGroth}, we can view $\Gr$ as an object with a $W$-action. Since $W_L\subseteq W$, we may identify $W_L$ with a subset of $\End(\Gr)$ via the above isomorphism. With this in hand, we define $W_L$-invariants $(\Gr)^{W_L}$. More concretely, $(\Gr)^{W_L}$ is the largest subobject of $\Gr$ so that for any endomorphism $f\in W_L\subseteq\End(\Gr)$, we have that $f|_{(\Gr)^{W_L}} = \mathrm{id}_{(\Gr)^{W_L}}$. In the same way, we can talk about $(\Sp)^{W_L}$. 


The next theorem is a consequence of the fact that the map from $\GRS$ to $\GRSP$ is a principle fibration with fiber $W_L$. See also \cite[Prop. 2.7]{BM}
\begin{theorem}\label{th1.9} We have isomorphisms
	\begin{enumerate}
		 \item $\GrP \cong (\Gr)^{W_L}$ and
		\item $\Spp\cong (\Sp)^{W_L}$. 
	\end{enumerate}
		\end{theorem}
\begin{proof}
First we prove part (1). By Lemma \ref{lm2.1}, there is a $W$-action on $\pirs_* \underline{\kp}_{\grs}$. This means we may consider the $W_L$-invariants $(\pirs_* \underline{\kp}_{\grs})^{W_L}$. As we have seen in Lemma \ref{lm1.6}, $\GRSP\cong \GRS/{W_L}$, and $\pirsp$ is obtained by factoring the covering map $\pirs: \GRS\to \grs$ through $\GRSP$. Therefore $\pirsp_*\underline{\kp}_{\GRSP}\cong (\pirs_* \underline{\kp}_{\GRS})^{W_L}$. Thus we may write \[\GrP\cong \IC(\grs, (\pirs_*\underline{\kp}_{\GRS}) ^{W_L} ).\]  Since $\mc{IC}$ is a fully faithful functor \cite[Lemma ~3.3.3]{Ac}, we get \[ \IC(\grs, (\pirs_*\underline{\kp}_{\GRS}) ^{W_L} )\cong \IC(\grs, \pirs_*\underline{\kp}_{\GRS})^{W_L}\cong (\Gr)^{W_L}.\]
	
Now part (2) follows from the first part and the fact that $\Spp\cong i_{\mc{N}}^*\GrP$.
\end{proof}

\section{The sheaf complex $\SpU$.}\label{sec3} Recall from Theorem \ref{th1.4} that $\Sp$ and $\Gr$ are related by Fourier Transform. The analogous relationship between $\Spp$ and $\GrP$ does not hold because under the isomorphism $\mf{g}\cong\mf{g}^*$, the spaces $\mf{p}$ and $\mc{N}_P$ are not orthogonal to each other. Thus we introduce another sheaf complex $\SpU$. This perverse sheaf has appeared in \cite{JMW2} where it played a major role to calculate $\mc{IC}$s on the nilpotent cone. The object $ \SpU$ is defined as ${\mu_{\mf{u}_P}}_*\underline {\kp}_{G\times^P \mf{u}_P}[\dim \mc{N}]$ where the map $\mu_{\mf{u}_P}: G\times^P \mf{u}_P\to \mc{N}$ is defined by $\mu_{\mf{u}_P}(g, u) = \mathrm{Ad}(g)u$.

\begin{lemma}
Let $\mf{p}$ be a parabolic subalgebra with cartan subalgebra $\mf{h}$ and nilradical $\mf{u}_P$. Then under the Killing form, $\mf{p}$ and $\mf{u}_P$ are orthogonal to each other. 

\end{lemma}

\begin{proof} Let $\mf{g}_\alpha$ denote the root-space associated to the root $\alpha$. 
	 For any two roots $\alpha,\beta$, 
	\[\langle \alpha, \beta\rangle =0\] if and only if $\alpha+\beta\neq0$.
	The parabolic Lie algebra decomposes as \[\mf{p}= \mf{h}\oplus (\oplus_{\alpha\in \Phi_L } \mf{g}_\alpha) \oplus( \oplus_{\alpha\in (\Phi^+ \backslash \Phi_L^+)  } \mf{g}_\alpha)\] while \[ \mf{u}_P=  \oplus_{\alpha\in (\Phi^+ \backslash \Phi_L^+)  } \mf{g}_\alpha.
\] Therefore it is clear that $\mf{p}$ is orthogonal to $\mf{u}_P$.
\end{proof}
The above orthogonality allows us to prove the next theorem.

\begin{theorem}\label{th1.10}
	There is an isomorphism,
\[T_\mf{g}(\GrP)[\dim \GP -\dim \mf{g}]  \cong {i_{\mc{N}}}_* \SpU[ \dim G\times^P \mf{u}_P-\dim \mc{N}].\]
\end{theorem}

\begin{proof}
From above lemma, $\mf{u}_P$ and $\mf{p}$ are dual to each other under the isomorphism
\[\mf{g}\to \mf{g}^*.\] This implies $\GP$ and $G\times^P \mf{u}_P $  are dual to each other as subvarieties of $G\times^P \mf{g}$. Let $i_{\GP}$ and $i_{G\times^P \mf{u}_P}$ denote the respective inclusion maps into $G\times^P \mf{g}$. We have the Fourier--Sato transform $T_{G\times^P \mf{g}}: D^b_{\con }(G\times^P \mf{g})\to D^b_{\con }(G\times^P \mf{g} )$. By \cite[Cor ~6.8.11]{Ac},
	\begin{equation}\label{eq1.1}
		T_{G\times^P \mf{g}} \circ (i_{\GP})_*(\underline{\kp}_{\GP}[\dim \GP])\cong (i_{G\times^P \mf{u}_P})_*\underline{\kp}_{G\times^P \mf{u}_P}[\dim G\times^P \mf{u}_P].
	\end{equation}
Consider the diagram	
	\begin{center}
\begin{minipage}{.2\textwidth}
	\begin{tikzcd}[scale=0.33] 
G\times^P \mf{g} \arrow[r, "\pi_G"] \arrow[d, "pr_1"]& \mf{g}\arrow[d, ]\\G/P	\arrow[r] & \{pt\}
	\end{tikzcd}.
\end{minipage}
	\end{center}
Again from \cite[Prop. 6.8.12]{Ac}, we get
\begin{equation*}
	(\pi_{G})_*\circ T_ {G\times^P \mf{g}}\circ (i_{\GP})_*(\underline{\kp}_{\GP}[\dim \GP])\cong  T_ {\mf{g}}\circ (\pi_{G})_*\circ (i_{\GP})_*(\underline{\kp}_{\GP}[\dim \GP]). \end{equation*}
 Using the above isomorphism and the diagram below, we get
 \begin{align}\label{1.2}
     (\pi_G)_* \circ T_ {G\times^P \mf{g} }\circ (i_{\GP})_*(\underline{\kp}_{\GP}[\dim \GP]) \cong T_{\mf{g}} \GrP[\dim \GP -\dim \mf{g}].
 \end{align}
 
\[\begin{tikzcd}
G\times^P \mf{g}	\arrow[r, "\pi_G"] & \mf{g}\\
 G\times^P \mf{p} \arrow[u, hook', "i_{\GP}"] \arrow[r, "\pi_P"] & \mf{g} \arrow[u, "="]
\end{tikzcd}
.\]
Now applying $(\pi_G)_*$ to both sides of \eqref{eq1.1} and combining with \eqref{1.2}, we get
\begin{align*}
T_ {\mf{g}}\GrP[\dim \GP -\dim \mf{g}] &\cong  (\pi_G)_* \circ T_ {G\times^P \mf{g} } \circ(i_{\GP})_*(\underline{\kp}_{\GP}[\dim \GP]) \text{ from \eqref{1.2}}
\\ & \cong (\pi_G)_* \circ (i_{G\times^P \mf{u}_P})_*(\underline{\kp}_{G\times^P \mf{u}_P}[\dim G\times^P \mf{u}_P]) \text{ from \eqref{eq1.1}} 
\\ & \cong (i_{\mc{N}})_*\circ (\mu_{G\times^P \mf{u}_P })_* (\underline{\kp}_{G\times^P \mf{u}_P}[\dim G\times^P \mf{u}_P]) 
\\ & = (i_{\mc{N}})_* (\SpU[ \dim G\times^P \mf{u}_P-\dim \mc{N}]).
\end{align*}
The last isomorphism holds because $\pi_G\circ i_{G\times^P \mf{u}_P}=i_{\mc{N}}\circ \mu_{\mf{u}_P}.$
\end{proof}

Theorem \ref{th1.10} yields the following corollary.
\begin{cor}\label{cor1.12} The Fourier transform induces an isomorphism
\[\End(\SpU)\cong \kp[W(L)].\]	
\end{cor}

This defines an action of $W(L)$ on $\SpU$. We denote the induced map by $\rhoU: \kp[W(L)]\to \End(\SpU)$.

For any parabolic subgroup $P$,
 the image of the map $\mu_{\mf{u}_P}$ is a   closure of a orbit. Lets call it $\bar{\mc{O}}$. 
We have two Cartesian squares in which the vertical maps are proper. 

\begin{figure}[!htb]
    \centering
    \begin{minipage}{0.45\textwidth}
        \centering
        \begin{tikzcd}
	G\times^P \mf{u}_P \arrow[r, hook] \arrow[d, "\mu_{\mf{u}_P}"]& G\times^P \mc{N}_P \arrow[d, "\mu_P"]\\
	\bar{\mc{O}} \arrow[r,hook, "i_{\mc{O}}"] & \mc{N} 
\end{tikzcd} 
    \end{minipage}\hfill
    \begin{minipage}{0.45\textwidth}
        \centering
        \begin{tikzcd}
G\times^P \mf{u}_P \arrow[r, hook] \arrow[d, "\mu_{\mf{u}_P}"] & G\times^P \mf{p}=\GP\arrow[d, "\pi_P"]\\
\bar{\mc{O}} \arrow[r, hook, "i_{\mf{g}}"] & \mf{g}	\end{tikzcd}
    \end{minipage}
\end{figure}

Proper Base Change Theorem with the above diagrams implies the next Lemma.

\begin{lemma}\label{th1.13} 
There are isomorphisms 
\begin{enumerate}
    \item $i_{\mc{O}}^* \Spp\cong \SpU$ and
    \item $i^*_\mf{g} \GrP\cong \SpU$.
\end{enumerate}
\end{lemma}

This defines another action of $W(L)$ on $\SpU $ through the restriction map $i_\mf{g}$. We denote the induced map by $\phiU: \kp[W(L)]\to \End(\SpU)$. Our aim is to study the relationship between these two actions in the next two sections.

\section{Parabolic induction}\label{sec4} The following diagram is used to define parabolic induction. 
\[\begin{tikzcd}
\mc{N}_L & \mc{N}_P	\arrow[l, "\pi"'] \arrow[r, "e"] & \NP = G\times^P \mc{N}_P
\arrow[r, "\mu_P"] &\mc{N}_G
\end{tikzcd}.
\]
Here, $\mc{N}_L$ is the nilpotent cone associated to $L$, $\pi$ is the projection, $e$ is the inclusion, and $\mu_P$ is the action map. The parabolic induction functor $\inn^G_P: D^b_L(\mc{N}_L,\kp)\to D^b_G(\mc{N}_G,\kp)$ is defined by \[\inn^G_P(\mc{F}):= (\mu_P)_! (e^* \f^G_P)^{-1}\pi^* \mc{F},
\] where $(e^* \f^G_P)^{-1}$ is the inverse to the equivariant induction equivalence $D^b_P(\mc{N}_P,\kp)\stackrel{\sim}{\to} D^b_G(\NP,\kp)$.
Along with the study of the Springer sheaf and generalized Springer correspondence \cite{Lu3}, \cite{Sp} in characteristic $0$, this functor has played a major role in modular Springer theory \cite{AJHR1}, \cite{AJHR2}, \cite{AJHR3}.

Let $\SpL$ be the Springer sheaf with respect to the Levi $L$. In the following theorem, we relate the partial Springer sheaf with the Springer sheaf assuming that $k$ is a field such that the nilpotent cone $\mc{N}_G$ is $k$-smooth. This holds in particular when characteristic of $k$ is 0 by \cite{BM}.

\begin{theorem} Assume that $k$ is a field such that the nilpotent cone $\mc{N}_G$ is $k$-smooth. Then the partial Springer sheaf $\Spp$ is a subobject of $\inn^G_P \SpL\cong\Sp$.	\end{theorem}
\begin{proof} , where $\ph^0$ represents perverse cohomology in degree 0
Our first claim is that the shifted constant sheaf $\underline{\kp}_{\mc{N}_G}[\dim \mc{N}_G]$ appears as a direct summand of $\Sp$. From the following diagram 
\[\begin{tikzcd}
\mu^{-1}(\mc{O}_{reg})	\arrow[r, "\tilde{j}"] \arrow[d, "\cong"]& \N \arrow[d, "\mu"] \\ \mc{O}_{reg} \arrow[r, "j"'] & \mc{N}_G
\end{tikzcd}\]
we can see $j^*\Sp \cong \underline{\kp}_{\mc{O}_{reg}}[\dim \mc{N}_G]$. Therefore

\begin{align*}
	0\ne \kp\cong \Hom (\underline{\kp}_{\mc{O}_{reg}} [\dim \mc{N}_G], \underline{\kp}_{\mc{O}_{reg}} [\dim \mc{N}_G])& \cong  \Hom (\underline{\kp}_{\mc{O}_{reg}} [\dim \mc{N}_G], j^* \Sp)\\&\cong \Hom(j_! \underline{\kp}_{\mc{O}_{reg}} [\dim \mc{N}_G], \Sp).
\end{align*} Thus we have a nonzero map from $j_! \underline{\kp}_{\mc{O}_{reg}} [\dim \mc{N}_G] $ to $\Sp$. Now our aim is to show this map induces an injective map from $\ph^0(j_! \underline{\kp}_{\mc{O}_{reg}} [\dim \mc{N}_G])$ to $\Sp$.

Let $(^p\mc{D}^{\geq0}, {^p\mc{D}^{\leq 0}})$ be the perverse $t$-structure on $D^b_G(\mc{N},\kp)$.
The functor $j_!$ is right exact so \[j_! \underline{\kp}_{\mc{O}_{reg}} [\dim \mc{N}_G] \in \pd^{\le 0}.\]
Consider the truncated distinguished triangle
\[\pt^{\le -1}( j_! \underline{\kp}_{\mc{O}_{reg}} [\dim \mc{N}_G])\to j_! \underline{\kp}_{\mc{O}_{reg}} [\dim \mc{N}_G] \to \ph^0( j_! \underline{\kp}_{\mc{O}_{reg}} [\dim \mc{N}_G])\to .
\] Apply $\Hom(-,\Sp)$ to get an exact sequence 
\begin{align*}
&\to \Hom(\pt^{\le -1}( j_! \underline{\kp}_{\mc{O}_{reg}} [\dim \mc{N}_G])[1], \Sp) \to \Hom(\ph^0( j_! \underline{\kp}_{\mc{O}_{reg}} [\dim \mc{N}_G]), \Sp) 
\\ &\to \Hom(j_! \underline{\kp}_{\mc{O}_{reg}} [\dim \mc{N}_G], \Sp ) \to \Hom(\pt^{\le -1}( j_! \underline{\kp}_{\mc{O}_{reg}} [\dim \mc{N}_G]), \Sp)\to .
\end{align*}
Now $\Hom(\pt^{\le -1}( j_! \underline{\kp}_{\mc{O}_{reg}} [\dim \mc{N}_G]),\Sp) $ and $\Hom(\pt^{\le -1}( j_! \underline{\kp}_{\mc{O}_{reg}} [\dim \mc{N}_G])[1],\Sp) $ are $0$ by the definition of $t$-structure. Hence we get \[0\ne\Hom(j_! \underline{\kp}_{\mc{O}_{reg}} [\dim \mc{N}_G],\Sp )\cong \Hom(\ph^0( j_! \underline{\kp}_{\mc{O}_{reg}} [\dim \mc{N}_G]),\Sp).\] This means we have a nonzero map from $\ph^0(j_! \underline{\kp}_{\mc{O}_{reg}} [\dim \mc{N}_G]) $ to $\Sp$. 

By \cite[Prop. ~5.1]{AM}, $\ph^0(j_! \underline{\kp}_{\mc{O}_{reg}}[\dim \mc{N}_G] )\cong \underline{\kp}_{\mc{N}_G}[\dim \mc{N}_G]$. Since $\mc{N}_G$ is $\kp$-smooth, $\underline{\kp}_{\mc{N}_G} [\dim \mc{N}_G]$ is isomorphic to the intersection cohomology $\IC(\mc{O}_{reg}, \underline{\kp}_{\mc{O}_{reg}})$ which is a simple perverse sheaf. Hence we have an injective map $\underline{\kp}_{\mc{N}_G} [\dim \mc{N}_G]\rightarrow \Sp$, and so $\underline{\kp}_{\mc{N}_G} [\dim \mc{N}_G]$ is a subobject of $\Sp$. Since induction is exact, we get $\inn^G_P\underline{\kp}_{\mc{N}_L}[\dim \mc{N}_L]\cong \Spp$ is a subobject of $\inn^G_P \SpL$ as well.

Finally sheaf functor calculations with the following diagram imply the isomorphism $\inn^G_P\SpL\cong \Sp$.
\[\begin{tikzcd}
\mc{N}_L & \mc{N}_P \arrow[l, "\pi_P"] \arrow[r, "e_P"'] & G\times^P \mc{N}_P \arrow[r, "\mu_P"'] & \mc{N}_G \\
P\times^B \mf{u}/\mf{u}_P \arrow[u, "\mu_L"]& P\times^B \mf{u} \arrow[l, "\pi"] \arrow[r, "e"] 	\arrow[u] & G\times^B \mf{u} \arrow[r, "\mu"] \arrow[u] & \mc{N}_G \arrow[u, "="]
\end{tikzcd}
\]Here $P\times^B \mf{u}/\mf{u}_P\cong L\times^{(B\cap L)} \mf{u}_{B\cap L}\cong \tilde{\mc{N}}_L$. Therefore $\Spp$ is a subobject of $\Sp$.
\end{proof}

\section{Two actions of $W(L)$}\label{sec5}
Corollary \ref{cor1.12} gives an isomorphism $\End(\SpU)\cong \kp[W(L)]$. We call this map induced by Fourier transform
\[\rhoU: \kp[W(L)]\to \End(\SpU).\] 
Also from Lemma \ref{th1.13}, we get another action of $W(L)$ on $\SpU$ induced by restriction which we denote by
\[\phiU: \kp[W(L)]\to \End(\SpU).\] The aim of this section is to analyze the connection between $\rhoU$ and $\phiU$.
 
\subsection{Actions on $\h^*(G/P)$}

Let $r$ be the projection $G/L\to G/P$.  Since $r$ is a vector bundle with fiber $U_P$, $r^*$ is an isomorphism
\[r^*: \h^*(G/P)\to \h^*(G/L).\]
Now $W(L)=N_G(L)/L$ acts on $G/L$ and so also acts on $\h^*(G/L)$. Therefore $W(L)$ acts on $\h^*(G/P)$ by the above isomorphism. Suppose $G/P$ has real dimension $2N$.

\begin{theorem}\label{th5.1}
	As a $W(L)$-module, $\h^*(G/P)$ is faithful as long as $char(\kp)$ does not divide $|W|$.
\end{theorem}

\begin{proof}
    Our first claim is that
\[\h^*(G/P)=\h^*(G/B)^{W_L}.\]
  The W action on the chohomology of G/B comes from the isomorphism $\h^*(G/B)\cong \h^*(G/T)$ and the $W(L)$-action on $\h^*(G/P)$ comes from the isomorphism $\h^*(G/P)\cong \h^*(G/L)$. Since $W_L$ permutes the roots associated to $L$ and fixes all others, we get $\h^*(G/L)\cong \h^*(G/T)^{W_L}$. This implies our above claim. Now consider $\h^*(G/B)$ as $N_W(W_L)$-representation. Since $N_W(W_L)=W(L)\ltimes W_L$, we have that $W(L)$ acts on $\h^*(G/B)^{W_L}$. If $\h^*(G/P)$ is an unfaithful $W(L)$-representation, then the $W$-action on $\h^*(G/B)$ must also be unfaithful. This is a contradiction by \cite[Lemma ~4.1,4.2]{AJHR1}. 
\end{proof}

Now we define two actions of $W(L)$ on $\h^*(G/P)$. The goal is to find a connection between these two actions. There is a natural action of $W(L)$ on $G/L$. Through the projection, $r: G/P\to G/L$ we can get an action of $W(L)$ on $\h^i(G/P)$. Now we will define another action of $W(L)$ on $\h^i(G/P)$. Recall $G/P$ is compact, therefore $\h^i_c(G/P)=\h^i(G/P)$. However the action of $W(L)$ on $\h^i_c(G/P)$ comes from the identification

\[\h^i_c(G/P)^*\cong \h^{2N-i}(G/P),\]

\noindent which we explain below in more detail.

There is an isomorphism 
\begin{equation}\label{eq5.7}
    \h^i_c(G/P)^*\cong \h^{BM}_i(G/P)
\end{equation} between the dual of compactly supported cohomology of $G/P$ and the Borel-Moore homology. When $k=\mathbb{Z}$, this holds because $\h^i(G/P)$ is a free abelian group for all $i$.
  We have the following $W(L)$-equivariant isomorphisms:
  \begin{align*}
  \h^i(G/P)^*\cong &	\h^{BM}_i(G/P) \text{, using $G/P$ is compact and (\ref{eq5.7}).}\\ \cong & \h^{BM}_{2N+i}(G/L)\\ & \cong  \h^{4N-(2N+i)}(G/L) 
 \\ & =\h^{2N-i}(G/L).\\
  \end{align*}
\noindent Here the second isomorphism uses again that $r$ is a vector bundle, and the third isomorphism follows from Poincare duality. 

  Recall that $\h^{2N-i}(G/L)\cong \h^{2N-i}(G/P)$, so from above the action of $W(L)$ on $\h^i(G/P)^*$ is the same as the action of $W(L)$ on $\h^{2N-i}(G/L)$.

\begin{theorem}\label{th3.1}
 With the two actions of $W(L)$ on $\h^i(G/P)$ and $\h^{2N-i}(G/P)$ defined above,
we	 have the following $W(L)$-equivariant isomorphism,
	\[\h^{2N-i}(G/P) \cong \h^i(G/P)^* \otimes \h^{2N}(G/P),\] where $\h^i(G/P)^*$ denotes the dual of $\h^i(G/P)$.
\end{theorem}

\begin{proof} The following diagram is $W(L)$-equivariant and commutative.
	\[\begin{tikzcd}
		\h^i(G/P)\times \h^{2N-i}(G/P) \arrow[d, "r^*"] \arrow[r, ] & \h^{2N}(G/P)\arrow[d, "r^*"]\\
		\h^i(G/L)\times \h^{2N-i}(G/L)  \arrow[r, ] & \h^{2N}(G/L)
	\end{tikzcd}\]	
As $G/P$ is compact, the upper  horizontal line is an isomorphism by Poincare duality. To prove this theorem we want this isomorphism to be $W(L)$-equivariant. The lower horizontal map is already $W(L)$-equivariant, hence   is a perfect pairing. Therefore we get that

\begin{equation}\label{eq3.2}
	\h^{2N-i}(G/L)\cong  \h^{i} (G/L)^* \otimes \h^{2N}(G/L)
\end{equation}

\noindent is a $W(L)$-equivariant isomorphism.

  From \eqref{eq3.2} and the above commutative diagram, we can see the two $W(L)$-actions defined on $\h^i(G/P)$ and $\h^{2N-i}(G/P)$ are connected in the following way, 
\begin{align*}
\h^{2N-i}(G/P) & \cong \Hom(\h^i(G/P), \h^{2N}(G/P))\\ & \cong \h^i(G/P)^*\otimes \h^{2N}(G/P) \text{ , as $W(L)$ representation. } 	
\end{align*}\end{proof}

\subsection{Comparison of $W(L)$-actions}

 We have already seen that we can define a $W(L)$-action on $\h^i(G/P)$ or on $\h^i_c(G/P)^*\cong \h^{2N-i}(G/P)$ from the $W(L)$-action on $G/L$ with the vector bundle map $r: G/L\to G/P$. We will call this action the classical action of $W(L)$. In this subsection, we define another action of $W(L)$ on $\h^i(G/P)$ (resp. on  $\h^i_c(G/P)$) coming from the $W(L)$-action on $\GrP$. We call this action the partial Grothendieck action. The goal of this subsection is to compare the classical action and the partial Grothendieck action.
 
For a variety $X$, let $a_X$ be the constant map $X\to \{pt\}$. The map $t:\GP\to G/P$ is a vector bundle with fiber $\mf{p}$, so we have the following isomorphisms
\begin{align}\label{4.4}
	\h^*({a_{\mf{g}}}_! \GrP)\xrightarrow{\cong} \h^*({a_{\GP}}_!\underline{\kp}[\dim \mf{g}])\xrightarrow{\cong} \h^{* +\dim \mf{g}-\dim \mf{p}}_c(G/P). 
\end{align}
\noindent This induces a map on the endomorphism rings
\[\gamma_c^P: \End(\GrP) \to \End( \h^{* +\dim \mf{g}-\dim \mf{p}}_c(G/P)).  
\]
We also have the following Cartesian square

\[\begin{tikzcd}
G/P \arrow[r, hook, "j"] \arrow[d, ] & \GP\arrow[d, "\pi_P"]\\ \{0\} \arrow[r, "j_0"] & \mf{g}
\end{tikzcd}.\]

\noindent Thus, we have an isomorphism 
\begin{equation}\label{4.5}
	\h^*(j_0^* \GrP)\xrightarrow{\cong}\h^{*+\dim \mf{g}}(G/P).
\end{equation}

\noindent Again, we get an induced a map on the endomorphism rings
\[\gamma^P: \End(\GrP) \to \End(\h^{*+\dim \mf{g}}(G/P)).\]
The proof of the following theorem is inspired by \cite[Proposition ~5.4]{Sh}.
\begin{theorem} 
	The classical action of $W(L)$ on $\h^i(G/P)$ is same as the action coming from partial Grothendieck.
\end{theorem}

\begin{proof}
	Recall that the classical action of $W(L)$ comes from the action on $G/L$ defined by
	\[w\cdot (gL)=gwL.\] The classical action of $W(L)$ on $\h^i(G/P)$ gets induced by classical action of $W(L)$ on $G/L$ and the vector bundle map, $r:G/L\to G/P$.
	For any $x\in \mf{g}$, let $\mc{P}_x$ be the inverse image of $x$ under the map $r$. Hence $\h^i(\GrP)_x$ is canonically isomorphic to $\h^{i+\dim \mf{g}}(\mc{P}_x,\kp)$ , and the partial Grothedieck action of $W(L)$ on $\h^i(\GrP)_x$ induces a corresponding action on $\h^{i+\dim \mf{g}}(\mc{P}_x,\kp)$. Therefore for $x=0$, we get the partial Grothedieck action on $\h^{i+\dim \mf{g}}(G/P)$.
	
	Now we assume $\zeta \in \lie(T)$ is a semisimple element, i.e $Z_G(\zeta)=T$. Let $\mc{O}$ be a $G$-orbit of $\zeta$ in $\mf{g}$.
	Let $p$ denote the projection $\pi_P^{-1}(\mc{O})\to G/P$. The idea of the proof is to show the map $p^*$ is $W(L)$-equivariant with respect to both the classical action and the partial Grothendieck action, and that it is an injective map. 
 
  Our first claim is that the induced map
	 \[p^*: \h^i(G/P)\to \h^i(\pi_P^{-1}(\mc{O}))\] is $W(L)$-equivariant with respect to the partial Grothendieck action of $W(L)$. To show this, we 
	 consider the maps
	 \[i': \pi_P^{-1}(\mc{O})\hookrightarrow \GP \hspace{.3cm}\textup{ and }\hspace{.3cm} p': \GP\to G/P.\]
	 Then $p=p'\circ i'$. Note that if $X\subset Y$ are subvarieties of $\mf{g}$, then the map $\h^i(\pi_P^{-1}(Y))\to \h^i(\pi_P^{-1}(X))$ is $W(L)$-equivariant with respect to the partial Grothendieck action. Thus, the pull-back $(i')^*$  is $W(L)$-equivariant. So it is enough to show $(p')^*$ is $W(L)$-equivariant. Now consider the map 
	 \[q: G/P{\cong} \pi_P^{-1}(0)\xhookrightarrow{} \GP.\] We know $q^*$ is $W(L)$-equivariant with respect to the partial Grothendieck action. Also $p'\circ q= Id_{G/P}$ and hence $p'^*=(q^*)^{-1}$. Therefore $(p')^*$ is $W(L)$-equivariant with respect to the partial Grothendieck action, and so $p^*$ is as well.  
	
 Consider the following set
	 \[\G^L_{rs}=\{(gL,x)\in G/L\times \mf{g}| \mathrm{Ad}(g^{-1})x\in \mf{l}_{rs}\}\]
	 with $L_{rs}=L\cap G_{rs}$, $G_{rs}$ being the collection of all regular semi-simple element in $G$ and $\mf{l}_{rs}=\lie(L_{rs})$. The group $W(L)$ acts on $\G^L_{rs}$ by 
	 \[w(gL,x)=(gwL,x).\] We define a map $\phi$ from $\G^L_{rs}$ to ${\pi_P}^{-1}(\grs)$ by
	 \[(gL,x)\to (gP,x).\] This map is an isomorphism from \cite[~ 2.4(iii)]{Sh}. 
	 With respect to this action and the trivial action on $\mf{g}_{rs}$, the map $\pi^P_{rs}: \G^L_{rs}\to \mf{g}_{rs}$ which is the restriction of $\pi_P$ on $\pi_P^{-1}(\grs)\cong \G^L_{rs}$, is $W(L)$ equivariant. Note, $\mc{O}$ is contained in $\mf{g}_{rs}$. 
	  Let $\tilde{\mc{O}}$ be the subvariety of $ \G^L_{rs}$ corresponding to   ${\pi_P}^{-1}(\mc{O})$. Then,
	 \[\tilde{\mc{O}}=\{(gwL, Ad(g)\zeta)\in G/L\times \mf{g}| g\in G, w\in W(L)\}.\]
	 Cleary we have an action of $W(L)$ on $\tilde{\mc{O}}$ given by
	 \[v(gwL,Ad(g)\zeta)=(gwvL, Ad(g)\zeta).\] 
	 
	 This is the restriction of the action of $W(L)$ on $\G^L_{rs}$. Thus the induced action of $W(L)$ on $\h^i(\pi_P^{-1}(\mc{O}))\cong \h^i(\tilde{\mc{O}})$ is the same as the partial Grothendieck action of $W(L)$ on it.
	 
	 Now let $\tilde{p}: \tilde{\mc{O}}\to G/L$ be the projection on the first component. Then $\tilde{p}$ is $W(L)$-equivariant with respect to the classical action of $W(L)$ on $G/L$ considered before. We have the following commutative diagram
	 
	 \[\begin{tikzcd}
	 \tilde{\mc{O}} \arrow[d, "\tilde{p}"] \arrow[r,"\cong" ] & \pi_P^{-1}(\mc{O}) \arrow[d, "p"]
	 \\
	 G/L \arrow[r] & G/P	
	 \end{tikzcd}
\] where $G/L\to G/P$ is the natural projection. Since the classical action of $W(L)$ comes from the $W(L)$-action on $G/L$, we see
\[p^*: \h^i(G/P)\to \h^i(\pi_P^{-1}(\mc{O}))\]is  $W(L)$-equivariant with respect to the classical action of $W(L)$ on $\h^i(G/P)$ also. Hence to complete the proof it is enough to show $p^*$ is injective, or equivalently $\tilde{p}^*$ is injective. Now for $w\in W(L)$, we define
\[\tilde{\mc{O}}_w=\{(gwL, Ad(g)\zeta)\in G/L\times \mf{g}| g\in G\}.\] Then we have
\[\tilde{\mc{O}}=\sqcup_{w\in W(L)} \tilde{\mc{O}}_w,\] and each $\tilde{\mc{O}}_w$ is an irreducible component of $\tilde{\mc{O}}$. The projection map
\[\tilde{p}_w: \tilde{\mc{O}}_w\to G/L\] is an isomorphism. Also we have
\[{\tilde{p}_w}^*={i_w}^*\circ \tilde{p}^*\] with $i_w: \tilde{\mc{O}}_w\hookrightarrow \tilde{\mc{O}}$. This forces $p^*$ to be injective.
\end{proof}

\begin{cor}\label{cor5.4}\begin{enumerate}
	\item The $W(L)$-action on $\End(\GrP)$ induced by $\gamma^P$ coincides with the $W(L)$-action induced from the classical action of $W(L)$ on $\h^{* +\dim \mf{g}-\dim \mf{p}}_c(G/P) $.
	
	\item The $W(L)$-action on $\End(\GrP)$ induced by $\gamma_c^P$ coincides with the $W(L)$-action induced from the classical action of $W(L)$ on $\h^{* +\dim \mf{g}}_c(G/P) $.
\end{enumerate}
	
\end{cor}

\subsection{The comparison of the Springer and the partial Springer. }
Recall the Weyl group acts on $\Sp$ in two different ways. One action comes from the restriction map $i_{\mc{N}}: \mc{N}\xhookrightarrow{} \mf{g}$ and corresponds to the map
    \[\rho: \kp[W]\to \End(\Sp).\] Another is a consequence of the Fourier transform and corresponds to the map
    \[\phi: \kp[W]\to \End(\Sp).\] It has been proved in \cite{AJHR1} that these two maps are connected by
    \[\mathrm{sgn}: \kp[W]\to \kp[W],\] which is the sign change map. More precisely,
    \[\rho=\phi\circ \mathrm{sgn}.\]

The approach to prove this Theorem in \cite{AJHR1} is to first prove what they call ``the easy case": $\kp$ either embeds in a $\mathbb{Q}$-algebra or in an $\mathbb{F}_p$-algebra where $p$ does not divide $|W|$. In this case, their proof uses arguments already established. That is, they study the classical $W$-action on $\h^*(G/B, \kp)$ which is faithful as a $\kp[W]$-module. Since $\kp=\mathbb{Z}$ is included in this situation, \cite{AJHR1} use it to deduce the general case by extending scalars. Furthermore, to avoid a great deal of compatibility-type bookkeeping, they introduce $\Sp^{\kp}$ as a functor from the category of $\kp$-modules to the category of perverse sheaves on $\mc{N}$ and study the ring of natural endotransformations of this functor. 

It has been proved \cite[Lemma ~3.2]{AJHR1} that evaluation
\[\End(\Sp^{\kp})\to \End(\Sp(\kp))\] induces a $\kp$-algebra isomorphism. Therefore proving $\rho=\phi \circ \sgn$ can be reduced to proving
\begin{equation}\label{eq5.13}
    \rho^{\kp}=\phi^{\kp}\circ \sgn,
\end{equation} where $\phi^{\kp},\rho^{\kp}: \kp[W]\to \End(\Sp^{\kp})$ are defined similarly to $\phi$ and $\rho$ but as actions on functors. Now by \cite[Prop. ~3.6]{AJHR1}, for any $\kp$-algebra $\kp'$, there is a canonical homomorphism
\[\beta^{\kp}_{\kp'}:\End(\Sp^{\kp})\to \End(\Sp^{\kp'})\] such that the two diagrams

\begin{figure}[!htb]
    \centering
    \begin{minipage}{0.45\textwidth}
        \centering
        \begin{tikzcd}
	\kp[W] \arrow[r, "\rho^{\kp}"] \arrow[d]& \End(\Sp^{\kp}\arrow[d, "\beta^{\kp}_{\kp'}"])\\
	\kp'[W] \arrow[r,"\rho^{\kp'}"] & \End(\Sp^{\kp'})
\end{tikzcd} 
    \end{minipage}\hfill
    \begin{minipage}{0.45\textwidth}
        \centering
        \begin{tikzcd}
	\kp[W] \arrow[r, "\phi^{\kp}"] \arrow[d]& \End(\Sp^{\kp}\arrow[d, "\beta^{\kp}_{\kp'}"])\\
	\kp'[W] \arrow[r,"\phi^{\kp'}"] & \End(\Sp^{\kp'})
 \end{tikzcd}
    \end{minipage}
\end{figure}
\noindent are commutative. Since we may consider any $\kp$ as a $\mb{Z}$-algebra, this allows definition of the desired morphism

\begin{equation}\label{eq5.10}
    \End(\Sp^{\mb{Z}})\to \End(\Sp^{\kp}).
\end{equation}

\noindent Thus \eqref{eq5.13} is implied by the special case
\begin{equation}\label{eqzcase}\rho^{\mb{Z}}=\phi^{\mb{Z}}\circ \sgn.\end{equation}
    
Now we outline how to prove \eqref{eqzcase}. Assume that the characteristic of $\kp$ does not divide $|W|$ (so this includes the case $\kp=\mathbb{Z}$). With this assumption, $\h^{top}(G/B)$ is the sign representation of $W$, and $\h^*(G/B)$ is  faithful as a $W$-representation \cite[Lemma ~4.3]{AJHR1}. We also get the following commutative diagrams
          
    \begin{equation}
        \begin{tikzcd}\label{eqrhosigma}
        &\End(\Sp) \arrow[rd, "\sigma"]&\\ \kp[W] \arrow[r, "r"] & \End(\Gr) \arrow[u, "\tilde{\rho}"] \arrow[r, "\gamma"]  &  \End(\h^*(G/B))
    \end{tikzcd}
    \end{equation}
        \begin{equation}
        \begin{tikzcd}\label{eqphisigma}
        &\End(\Sp) \arrow[rd, "\sigma"]&\\ \kp[W] \arrow[r, "r"] & \End(\Gr) \arrow[u, "\tilde{\phi}"] \arrow[r,"\gamma_c"]  &  \End(\h^*(G/B))
    \end{tikzcd}. 
    \end{equation}
    Here the map $r$ is coming the isomorphism defined in (\ref{eq:endGroth}). Thus, we get $\sigma \circ \rho= \gamma \circ r$ and $\sigma\circ \phi= \gamma_c \circ r$. As classical $W$-representations, we get the isomorphism
    \[\h^*_c(G/B)\cong \h^*(G/B ) \otimes \h^{top}(G/B). \]
    Therefore we also get
    \[ \gamma \circ r=\gamma_c \circ r \circ \mathrm{sgn}.\]
Because $\h^*(G/B)$ is faithful as a $W$-representation, the maps $\sigma \circ \tilde{\phi}\circ r $ and $\sigma\circ \tilde{\rho} \circ r \circ \mathrm{sgn}$ are injective. Since $\tilde{\rho} \circ r $ is an
  isomorphism, $\sigma $ must be injective. This allows us to say $\tilde{\rho} \circ r = \tilde{\phi} \circ r \circ \mathrm{sgn}$ which is same as $\rho =\phi \circ \mathrm{sgn}$. 
  
By \cite[Theorem~3.7]{AJHR1}, the case $\kp=\mb{Z}$ is enough to prove the result in all characteristic. This same reasoning can be used in our situation to reduce the question to $\kp=\mb{Z}$.  In our case, we get a connection between two actions of $W(L)$ on $\SpU$. 
Let us denote the isomorphism induced from Theorem \ref{th1.10} by
\[\widetilde{\phi}^P: \End(\GrP) \to \End(\SpU),\]
and the map induced from Lemma \ref{th1.13} by
\[\widetilde{\rho}^P: \End(\GrP)\to \End(\SpU).\] Consider the Cartesian diagram
\[\begin{tikzcd}
G/P \arrow[r, "i_{G/P}"] \arrow[d] & G\times^P \mf{u}_P \arrow[r, "i_{\NP}"] \arrow[d, "\mu^{\mf{u}_P}"] & \tilde{\mf{g}}^P \arrow[d,"\pi^P"] \\ \{0\} \arrow[r, "i_0"] & \bar{\mc{O}} \arrow[r,"i_{\mc{O}}"] & \mf{g}
\end{tikzcd}.\]
By the base-change theorem we get
\[ \h^{*-2N}(i_0^* \SpU)\xrightarrow{\cong} \h^*(G/P,\kp).\] This induces a map between the endomorphism rings
\[\sigma^P: \End(\SpU)\to \End(\h^*(G/P,\kp)) \cong \End(\h^*_c(G/P,\kp)) \text{ by Theorem \ref{th3.1}}. \]

The following lemma is our analogue of \cite[Lemma ~4.7]{AJHR1}.

 \begin{lemma}
     We have $\gamma^P= \sigma^P\circ \widetilde{\rho}^P $ and $\gamma^P_c =\sigma^P\circ \widetilde{\phi}^P$.
 \end{lemma} 
 The proof follows in the same manner as Lemma 4.7 from \cite{AJHR1}. Once we have this lemma, we get commutative diagrams similar to \eqref{eqrhosigma} and \eqref{eqphisigma}:

 \begin{equation}\label{5.13a}
        \begin{tikzcd}
        &\End(\SpU) \arrow[rd, "\sigma^P"]&\\ \kp[W(L)] \arrow[r, "r^P"] & \End(\GrP) \arrow[u, "\widetilde{\rho^P}"] \arrow[r, "\gamma^P"]  &  \End(\h^*(G/P))
    \end{tikzcd}.
    \end{equation} 

    \begin{equation}\label{5.13b}
        \begin{tikzcd}
        &\End(\SpU) \arrow[rd, "\sigma^P"]&\\ \kp[W(L)] \arrow[r, "r^P"] & \End(\GrP) \arrow[u, "\widetilde{\phi^P}"] \arrow[r, "\gamma^P_c"]  &  \End(\h^*(G/P))
    \end{tikzcd}.
    \end{equation} 
\begin{lemma}\label{lm5.6}
 If $\kp$ embeds either in a $\mb{Q}$-algebra or in an $\mb{F}_p$-algebra where $p$ does not divide $|W|$, then $\sigma^P$ is injective.
\end{lemma}

\begin{proof}
From Theorem \ref{th5.1}, we know the maps
\[\gamma^P\circ r^P \text{ and } \gamma_c^P \circ r^P: \kp[W(L)]\to \End(\SpU)\] are injective. By the above commutative diagram, these maps are the same as $ \sigma^P \circ \widetilde{\rho}^P \circ r^P $ and $\sigma^P \circ \widetilde{\phi }^P \circ r^P$. As we already know $\widetilde{\phi}^P\circ r^P$ is an isomorphism by Corollary \ref{cor1.12}, therefore $\sigma^P$ is injective.
\end{proof}

In the Springer sheaf case, the map $\sgn$ established the connection between $\rho$ and $\phi$. 
 For the partial Springer sheaf, it is well-known that the top cohomology of $G/P$ is an one dimensional representation of $W(L)$. Therefore it determines a character $\epsilon_U: W(L)\to \kp^\times$. We get a corresponding map
 \[\Lambda_U: \kp[W(L)]\to \kp[W(L)]\] defined by $w\mapsto \epsilon_U(w)w.$

 \begin{theorem}
     We have the following equality
     \[\rhoU=\phiU \circ \Lambda_U.\]
 \end{theorem}

 \begin{proof}
     By Theorem \ref{th3.1} and Corollary \ref{cor5.4}, we get
     \[ \gamma^P_c\circ r^P =\gamma^P \circ r^P \circ \Lambda_U. \] By Diagrams \eqref{5.13a} and \eqref{5.13b}, the above equality is the same as
     \[ \sigma^P \circ \rhoU=\sigma^P \circ \phiU \circ \Lambda_U.\] Finally, by Lemma \ref{lm5.6}, we know $\sigma^P$ is injective. This forces the equality
     \[\rho_U =\phi_U\circ \Lambda_U.\]
\end{proof}

 \begin{ex} It is natural to consider whether the sign change theorem for the $W$-action on $\Sp$ somehow `restricts' to the analogue we've presented for $\SpU$. Let's suppose that $L$ is a Levi subgroup such that there is a cuspidal perverse sheaf on $\mc{N}_L$. In this case, the group $W(L)$ is a Coxeter group \cite[Theorem ~5.9]{Lu1}. Hence there is a sign representation of $W(L)$ and a corresponding map  \[\sgn: \kp[W(L)]\to \kp[W(L)].\]
 
In this case, this question amounts to: should we expect a diagram like the following
 \[
 \begin{tikzcd}
     \kp[W(L)] \arrow[r, "\sgn"] \arrow[d] & \kp[W(L)] \arrow[r] \arrow[d] & \End(\h^*(G/P)) \arrow[d]\\ \kp[W] \arrow[r, "\sgn"] & \kp[W] \arrow[r] & \End(\h^*(G/B)) 
 \end{tikzcd}
 \]
to commute?
We consider a simple example that proves the left square does not commute in general. 
Let $G=SL_4$ and $L=S(GL_2\times GL_2)$. It is well-known that $\mc{N}_L$ supports a cuspidal perverse sheaf. The Weyl group of $G$ is $W=S_3=\langle s_1,s_2, s_3\rangle$, and the Weyl group of $L$ is $W_L=S_2\times S_2=\langle s_1, s_3 \rangle$.
It is not hard to see that $W(L)=\langle s=s_2s_1s_3s_2 \rangle$. Note that $s$ is simple reflection in $W(L)$ which means it has length $1$, but it has length $4$ regarded as an element in $W$. In the above diagram, commutativity of the left square requires $s$ to have the same parity of length with respect to both $W$ and $W(L)$. Thus, this diagram does not commute. 
\end{ex}

\end{document}